\documentclass[a4paper,11pt]{article}
\usepackage{amsmath,amsthm,amsfonts,amssymb,color,enumerate,xfrac,tikz,comment,pgfplots, mathtools}
\usepackage[colorlinks,citecolor=blue,urlcolor=blue]{hyperref}
\usetikzlibrary{shapes.geometric,positioning,matrix}
\excludecomment{codes}
\usepackage[text={6.5in,9.0in},centering]{geometry}

\usepackage{graphicx,url}
\RequirePackage[OT1]{fontenc}
\usepackage{datetime}
\usepackage{subcaption}
\usepackage[normalem]{ulem} 
\usepackage{soul} 
\makeatother
\numberwithin{equation}{section}
\allowdisplaybreaks

\usepackage{amssymb}
\usepackage{pifont}
\theoremstyle{plain}
\newtheorem{theorem}{Theorem}[section]

\newtheorem{lemma}[theorem]{Lemma}

\theoremstyle{definition}
\newtheorem{definition}[theorem]{Definition}

\newtheorem{assumption}[theorem]{Assumption}

\newtheorem{remark}[theorem]{Remark}

\newtheorem{example}[theorem]{Example}

\newcommand{\E}{\mathbb{E}}

\newcommand{\ud}{\ensuremath{\mathrm{d} }}

\newcommand{\Norm}[1]{\left\|#1\right\|}

\newcommand{\Pro}{\mathbb{P}}

\newcommand{\R}{\mathbb{R}}

\usepackage[utf8]{inputenc}
\usepackage[strict=true,style=english]{csquotes}
\usepackage[backend=biber,
            style=alphabetic,
            natbib=true,
            maxbibnames=99,
            sorting=nyt,
            abbreviate=true]{biblatex}
\AtEveryBibitem{
 \clearfield{url}
 \clearfield{issn}
 \clearlist{location}
 \clearfield{month}
 \clearfield{series}

 \ifentrytype{book}{}{
  \clearlist{publisher}
  \clearname{editor}
 }
}
\title{
  Global solution for superlinear stochastic heat equation on $\R^d$\\
  under Osgood-type conditions
}
\author{
 Le Chen\footnote{Department of Mathematics and Statistics, Auburn University, Auburn, Alabama, USA. Email: \url{le.chen@auburn.edu}.}
 \and
 Mohammud Foondun\footnote{Department of Mathematics and Statistics, University of Strathclyde,
 Glasgow, UK. Email: \url{mohammud.foondun@strath.ac.uk}.}
 \and
 Jingyu Huang\footnote{School of Mathematics, University of Birmingham, Birmingham, UK. Email: \url{j.huang.4@bham.ac.uk}.}
 \and
 Mickey Salins\footnote{Department of Mathematics and Statistics, Boston University, Boston, Massachusetts, USA. Email: \url{msalins@bu.edu}.}
}
\date{\today}

\begin{document}
\maketitle

\begin{abstract}
  We study the \textit{stochastic heat equation} (SHE) on $\R^d$ subject to a
  centered Gaussian noise that is white in time and colored in space.The drift term is assumed to satisfy an Osgood-type
  condition and the diffusion coefficient may have certain related growth. We show that there exists random field solution which do not explode in finite time. This complements and improves upon recent results on blow-up of solutions to stochastic partial differential equations.

\end{abstract}

\setlength{\parindent}{1.5em}

\hypersetup{linkcolor=black}
\tableofcontents
\hypersetup{linkcolor=red}

\section{Introduction}

In this paper, we study the following superlinear \textit{stochastic heat
equation} (SHE) on $\R^d$:
\begin{equation}\label{E:SHE}
  \begin{dcases}
    \dfrac{\partial u(t,x)}{\partial t} = \dfrac{1}{2}\Delta u(t,x) + b\left(u(t,x)\right) + \sigma\left(u(t,x)\right)\dot{W}(t,x)\,, & t>0, \: x\in\R^d, \\
    u(0,\cdot) = u_0(\cdot),
  \end{dcases}
\end{equation}
where both $b$ and $\sigma$ are locally Lipschitz continuous, vanish at zero,
and may have superlinear growth at infinity. The noise $\dot{W}$ is a centered
Gaussian noise which is white in time and colored in space with the following
covariance structure
\begin{equation}\label{E:Noise}
  \E\left[\dot{W}(s,y)\dot{W}(t,x)\right] = \delta(t-s)f(x-y)\,,
\end{equation}
where $\delta$ is the Dirac delta measure and $f$ is a nonnegative and
nonnegative-definite function on $\R^d$. The case when $f = \delta_0$ refers to
the space-time white noise. We note that $f$ induces an inner product
\begin{equation}\label{E:H-inner-Prod}
  \langle \varphi, \psi\rangle_{H} = \iint_{\R^{2d}} \varphi(y)\psi(z) f(y-z) \ud y \ud z
\end{equation}
for $\varphi, \psi \in C_c^{\infty}(\R^d)$, i.e., smooth functions with compact
supports in $\R^d$, and define the Hilbert space $H$ to be the completion of
$C^{\infty}_c(\R^d)$ under this inner product.

The solution to~\eqref{E:SHE} is understood in the mild formulation:
\begin{equation}\label{E:Mild}
  \begin{aligned}
    u(t,x) = (p_t*u_0)(x) & +\int_0^t \int_{\R^d} p_{t-s}(x-y)b(u(s,y))\ud y \ud s \\
                          & +\int_0^t \int_{\R^d} p_{t-s}(x-y)\sigma(u(s,y))W(\ud s, \ud y)\,,
  \end{aligned}
\end{equation}
where the stochastic integral is the \textit{Walsh
integral}~\cite{walsh:86:introduction,dalang.khoshnevisan.ea:09:minicourse},
$p_t(x) = \left(2\pi t\right)^{-d/2}\exp\left(-|x|^2 /2t\right)$ is the heat
kernel, and ``$*$" denotes the convolution in the spatial variable. \bigskip

For the one-dimensional deterministic ordinary differential equation
\begin{align}\label{E:ODE}
  \dfrac{dv}{dt} = b(v(t)) \quad \text{with $b \geq 0$ and $v(0)=c>0$,}
\end{align}
the \textit{Osgood condition}~\cite{osgood:98:beweis} characterizes finite time
explosion. Solutions explode in finite time if and only if the following
\textit{finite Osgood condition} holds:
\begin{align}\label{E:Osgood-Fnt}
  \int_c^\infty \frac{1}{b(u)}\ud u< +\infty.
\end{align}
This can be seen by rewriting~\eqref{E:ODE} as follows
$t=\int_0^t\frac{1}{b(u(s))} \ud u(s) .$ After a change of variable, we obtain
\begin{equation*}
    t=\int_{v(0)}^{v(t)}\frac{1}{b(s)} \ud s.
\end{equation*}
This shows that \eqref{E:Osgood-Fnt} is both necessary and sufficient for
blow-up of solutions to~\eqref{E:ODE}. The Osgood condition does not fully
characterize finite time explosion for deterministic partial differential
equations as demonstrated by the famous example provided by
Fujita~\cite{fujita:66:on}.

Finite time explosion for \textit{stochastic partial differential equations}
(SPDEs) with superlinear $b$ and $\sigma$ have only recently gained some
attention. The case of a bounded spatial domain has received more attention.
Bonder and Groisman~\cite{fernandez-bonder.groisman:09:time-space} demonstrated
that if $b$ satisfies~\eqref{E:Osgood-Fnt}, then a one-dimensional stochastic
heat equation with additive space-time white noise will always explode. Dalang,
Khoshnevisan, and Zhang~\cite{dalang.khoshnevisan.ea:19:global} established that
{for the same equation but with multiplicative space-time white noise, global
solutions exist} if $b$ grows no faster than $u \log u$ and $\sigma$ grow slower
than $u (\log u)^{1/4}$. Foondun and Nualart~\cite{foondun.nualart:21:osgood}
showed that, analogous to deterministic ordinary differential equations, the
Osgood condition~\eqref{E:Osgood-Fnt} on $b$ characterizes finite time explosion
for the stochastic heat equation with additive noise on bounded domains in any
spatial dimension. In other words, the additive-noise stochastic heat equation
with bounded initial data will explode in finite time if and only if $b$
satisfies~\eqref{E:Osgood-Fnt}. In contrast, Salins~\cite{salins:22:global}
demonstrated that if $b$ does not satisfy condition~\eqref{E:Osgood-Fnt}, i.e.,
$b$ satisfies the following \textit{infinite Osgood condition} holds:
\begin{align}\label{E:Osgood-Inf}
  \int_c^\infty \frac{1}{b(u)}\ud u= +\infty \quad \text{for all $c>0$,}
\end{align}
then to guarantee the existence of global solutions (i.e., solutions of all
time), one can allow $\sigma$ to grow superlinearly as long as it satisfies an
appropriate Osgood-type assumption. {Shang and
  Zhang~\cite{shang.zhang:22:stochastic} studied a superlinear stochastic heat
equation on a bounded domain driven by a Brownian motion (namely,
space-independent white noise).}

Finite time explosion for the superlinear stochastic wave equations has been
investigated in~\cite{foondun.nualart:22:non-existence}
and~\cite{millet.sanz-sole:21:global}, which proved sufficient conditions for
finite time explosion and those for global solutions, respectively. In both
works, the compact support property, which is inherited from the fundamental
solution of the wave equation, plays a crucial role.

The question of finite time explosion for the stochastic heat equation on
unbounded spatial domains is more complicated because solutions to the
stochastic heat equation can be unbounded in space in the sense that $\Pro\left(
\sup_x |u(t,x)|=+\infty \right)=1$ for any $t>0$. Shang and
Zhang~\cite{shang.zhang:21:global} showed that if $b$ grows like $u\log(u)$ and
if $\sigma$ is bounded and Lipschitz, then there exist global solutions to the
stochastic heat equation on $\R$. Salins~\cite{salins:22:existence} proved an
Osgood-type assumption on $b$, allowing for faster growth than $u\log u$,
implies the existence of global solutions when $\sigma~\equiv 1$. Chen and
Huang~\cite{chen.huang:23:superlinear} investigated the existence of global
solutions to the stochastic heat equation defined on an unbounded spatial domain
under assumptions that guarantee that the solutions are bounded in space. The
current paper is a major improvement over the results
of~\cite{chen.huang:23:superlinear}, examining scenarios where solutions remain
spatially bounded, while allowing for more general assumptions that accommodate
faster growth of the superlinear $b$ and $\sigma$ terms. \bigskip

For the stochastic noise, one commonly assumes the following strengthened
\textit{Dalang's condition}:
\begin{align}\label{E:Dalang+}
  \Upsilon_\alpha\coloneqq (2\pi)^{-d}\int_{\R^d} \frac{\widehat{f}(\xi)\ud\xi}{(1+|\xi|^2)^{1-\alpha}} < \infty\,,
  \quad \text{for some $0 < \alpha < 1$,}
\end{align}
where $\widehat{f}(\xi)$ is the Fourier transform of $f$, namely,
$\widehat{f}(\xi) = \mathcal{F}f(\xi) = \int_{\R^d} f(x) e^{-i x\cdot \xi} \ud
x$. When $\alpha = 0$, it reduces to a weaker \textit{Dalang's
condition}~\cite{dalang:99:extending}:
\begin{align}\label{E:Dalang}
  \Upsilon(\beta)\coloneqq (2\pi)^{-d}\int_{\R^d} \frac{\widehat{f}(\ud \xi)}{\beta+|\xi|^2}<+\infty \quad \text{for some and hence for all $\beta>0$.}
\end{align}
In order to get slightly stronger results, instead of
condition~\eqref{E:Dalang+}, we make the following slightly weaker assumption on
the noise:
\begin{assumption}\label{A:noise}
  There exists $\alpha \in (0,1]$ such that
  \begin{equation}
    \limsup_{s \downarrow 0} s^{1-\alpha} \int_{\mathbb{R}^d} e^{-s|\xi|^2}\widehat{f}(\xi) \ud \xi <+\infty.
  \end{equation}
\end{assumption}
One can show that the strengthened Dalang's condition~\eqref{E:Dalang+} implies
Assumption~\ref{A:noise}. In the one-dimensional space-time white noise setting,
Assumption~\ref{A:noise} is satisfied with $\alpha = 1/2$ but the strengthened
Dalang's condition~\eqref{E:Dalang+} is satisfied only when $\alpha \in
(0,1/2)$. \bigskip

Regarding the drift term $b$ and the diffusion coefficient $\sigma$, we make the
following Osgood-type assumptions following~\cite{salins:22:global*1}.

\begin{assumption}\label{A:Osgood}
  Assume that
  \begin{enumerate}
    \item Both $b$ and $\sigma$ are locally Lipschitz continuous;
    \item $b(0) = 0$ and $\sigma (0) = 0$;
    \item There exists a positive, increasing function
      $h:[0,\infty)\to[0,\infty)$ such that:
      \begin{enumerate}
        \item (Superlinear growth) $\R_+ \ni u\to u^{-1}h(u)$ is non-decreasing
          with $\displaystyle \lim_{u\to 0} u^{-1}h(u) \ge
          \exp\left(1/\alpha\right)$;
        \item (Osgood-type condition of the infinite type) $\displaystyle \int_1^\infty
          \frac{1}{h(u)} \ud u = +\infty$;
        \item For all $u\in \R$, $|b(u)| \le h(|u|)$;
        \item For all $u\in\R$, it holds that
          \begin{equation}\label{E:new-osgood}
            |\sigma(u)| \leq |u|^{1- \alpha/2} (h(|u|))^{\alpha/2} \left(\log\left(\frac{h(|u|)}{|u|}\right)\right)^{-1/2};
          \end{equation}
     \end{enumerate}
    where the constant $\alpha$ in parts (a) and (d) is given in
    Assumption~\ref{A:noise}.
  \end{enumerate}
\end{assumption}

\begin{remark}
  In~\cite{salins:22:global}, Salins assumed that $|\sigma(u)| \leq
  |u|^{1-\gamma} (h(|u|))^\gamma$ for some $\gamma < \alpha/2$ (where the paper
  uses the notation $1-\eta = \alpha$). Our proof is based on exponential tail
  estimates; see Lemma~\ref{L:exponential-estimate} below. The condition on
  $\sigma$, given by~\eqref{E:new-osgood}, allows for a faster growth rate for
  $\sigma$ than that in~\cite{salins:22:global}. The arguments based on the
  exponential tail estimates can also be applied in the finite domain setting.
\end{remark}

Let us introduce some notation. For $p\ge 1$, set
\begin{align*}
  V_p\coloneqq L^p (\R^d)\cap L^\infty(\R^d) \quad\text{and}\quad
  \Norm{\cdot}_{V_p} \coloneqq \max\left(\Norm{\cdot}_{L^p (\R^d)},\: \Norm{\cdot}_{L^\infty(\R^d)}\right).
\end{align*}
To guarantee that solutions to~\eqref{E:SHE} remain bounded in space, we make
the following assumption on the initial data.

\begin{assumption}\label{A:initial}
  The initial data $u_0\in V_p$ for some $p \geq 2$.
\end{assumption}

\begin{remark}
  The space of bounded $L^p(\R^d)$-functions, namely, $V_p$, is monotone
  increasing with respect to $p$, i.e., $\phi\in V_p$ implies that $\phi\in V_q$
  for all $q\ge p\ge 1$; see Lemma~\ref{L:Lr} below. In
  Assumption~\ref{A:initial}, we require $p\geq 2$ because we need to apply the
  \textit{Burkholder-Davis-Gundy (BDG) inequality} in Banach space; see
  Appendix.
\end{remark}

The aim of this present paper is to prove the following theorem:

\begin{theorem}\label{T:Main}
  Suppose that the noise satisfies Assumption~\ref{A:noise} with some $\alpha\in
  (0,1]$, the initial condition satisfies Assumption~\ref{A:initial} for some
  $p\ge (2+d)/\alpha$, and that $b(\cdot)$ and $\sigma(\cdot)$ satisfy
  Assumption~\ref{A:Osgood}. Then, we have the following:
  \begin{enumerate}
    \item There exists a unique mild solution $u(t,x)$ to~\eqref{E:SHE} for all
      $(t,x)\in (0,+\infty)\times\R^d$.
    \item Moreover, if $f$ satisfies the strengthened Dalang's
      condition~\eqref{E:Dalang+} with some $\alpha\in (0,1]$, then the solution
      $u(t,x)$ is H\"older continuous: $u\in C^{\alpha/2 -,\, \alpha
      -}\left((0,T]\times\R^d\right)$ a.s., where $C^{\alpha_1 -, \,
      \alpha_2-}\left(D\right)$ denotes the H\"older continuous function on the
      space-time domain $D$ with exponents $\alpha_1-\epsilon$ and
      $\alpha_2-\epsilon$ in time and space, respectively, for any small
      $\epsilon>0$.
  \end{enumerate}
\end{theorem}

The conditions that $b(0)=0$, $\sigma(0)=0$ and $u_0 \in V^p$ guarantee that the
solution to~\eqref{E:SHE} remains in $V^p$ almost surely. This allows us to
perform a localization procedure on the solution. We can prove that solution
cannot explode in finite time by proving that a sequence of hitting times
$\tau_n \uparrow +\infty$. \bigskip

Our main result---Theorem~\ref{T:Main}---provides the optimal condition on the
drift term $b$, which can be seen from the following theorem:

\begin{theorem}\label{T:Ex-Blowup}
  Let $b: \mathbb{R} \to \mathbb{R}$ be a locally Lipschitz continuous,
  nondecreasing and convex function that vanishes at zero ($b(0)=0$). Suppose
  that $\sigma(\cdot)$ vanishes at zero and is bounded, namely, $\sigma(0) = 0$
  and $\sup_{u \in \mathbb{R}}|\sigma(u)| \leq K$. Under the noise
  assumption---Assumption~\ref{A:noise}, if $b$ satisfies the finite Osgood
  condition~\eqref{E:Osgood-Fnt}, then for any $p\ge 2$, there exists some
  nonnegative initial condition $u_0(\cdot)\in V_p$ to~\eqref{E:SHE} such that
  solutions to~\eqref{E:SHE} will explode in finite time with positive
  probability.
\end{theorem}

\begin{remark}
  Theorem~\ref{T:Ex-Blowup} might hold true in certain cases without the
  condition that $\sigma$ is bounded. However, having this restriction
  simplifies its proof and is sufficient to demonstrate the optimality of our
  condition on $b(\cdot)$. Determining the optimal condition on the diffusion
  coefficient $\sigma$ is still an open problem, which will be left for further
  study.
\end{remark}

\begin{remark}
  It is interesting to note that the conclusion of Theorem~\ref{T:Ex-Blowup}
  cannot hold for any initial condition. Indeed, we can take $\sigma$ to be
  identically zero, so that~\eqref{E:SHE} becomes deterministic. Then, for a
  class of small initial conditions, one can then produce non-trivial solutions
  which does not explode in finite time. Indeed if we choose $b(u)=u^p$ with
  $p\geq 1+2/d$, then the solutions will not explode provided we choose the
  initial condition to be small enough.
  See~\cite{quittner.souplet:19:superlinear} for more information on this and
  other blow-up results for non-linear PDEs . In the proof of
  Theorem~\ref{T:Ex-Blowup}, it will be clear that one needs to take the initial
  condition large enough in a suitable sense for the solution to blow-up.
\end{remark}

In order to introduce some examples, we use the following notation for the
repeated logarithm function: $\log(u,1) \coloneqq \log(u)$ and for $k\ge 2$,
$\log(u,k)\coloneqq \log\left(\log u, k-1\right)$.

\begin{example}
  Examples of superlinear Osgood-type $h$ functions include $h(u) = u\log(u)$,
  $h(u) = u \log(u)\log(u,2)$, $h(u) = u \log(u)\log(u,2)\log(u,3)$, and so on,
  as listed in the first column of the tables in Table~\ref{Tb:Eg}. In
  particular, we have two special cases:
  \begin{itemize}
    \item Since $\int_{\mathbb{R}^d} e^{-s|\xi|^2}\ud\xi = (2\pi)^d
      \int_{\mathbb{R}^d} p_s^2(y)\ud y = (\pi /s)^{d/2}$, we see that when
      $d=1$, Assumption~\ref{A:noise} is satisfied with $\alpha = 1/2$.
      Therefore, in case of the spatial dimension one and $\dot{W}$ is
      space-time white noise, we can take $\alpha=1/2$ in~\eqref{E:new-osgood}.
      In particular, we have the concrete examples listed in
      Table~\ref{Tb:ST-White-1D}.
    \item If $f(0)<\infty$, or equivalently, $\int_{\mathbb{R}^d}
      \widehat{f}(\xi)\ud\xi<+\infty$, then $\alpha =1$. In this case, the
      growth rates for $\sigma$ and some typical $h$ are listed in
      Table~\ref{Tb:Bdd}.
  \end{itemize}
  In general, for $\alpha\in (0,1]$ given in Assumption~\ref{A:noise}, the
  function $h$ with repeated logarithms and the corresponding growth bound for
  $\sigma$ are listed in Table~\ref{Tb:General}.
\end{example}

In~\cite{chen.huang:23:superlinear}, it is demonstrated that there exists a
global solution when $b(u)$ grows as fast as $u \log(u)$ and $\sigma(u)$ grows
as fast as $u \left( \log u \right)^{\alpha/2}$. This paper utilizes a method
motivated by~\cite{dalang.khoshnevisan.ea:19:global}, which studied the setting
of a bounded one-dimensional domain. Our new result is stronger than the
previous ones because we can allow $b$ to grow faster than $u \log u$ as long as
$b$ is dominated by an Osgood-type $h$. We believe that the methods used
in~\cite{chen.huang:23:superlinear} cannot easily be extended to deal with $b$
growing faster than $u\log u$.

Interestingly, there remains a specific scenario where the method used
in~\cite{chen.huang:23:superlinear} yields a stronger result than
Theorem~\ref{T:Main} of the current paper. Specifically, when $b$ grows no
faster than $u\log u$ and $\sigma$ grows as $u (\log u)^{\alpha/2}$, the main
result of~\cite{chen.huang:23:superlinear} can be used to prove that solutions
never explode. Based on our current strategy of using stopping time arguments
and exponential estimates, we can let $\sigma$ grow like (see
Table~\ref{Tb:General} with $K=1$)
\begin{align*}
  u \left[ \log(u)\right ]^{\alpha/2} \left[ \log(u,2)\right ]^{-1/2}.
\end{align*}
However, if $\alpha<1$, then we cannot achieve $u (\log u)^{\alpha/2}$ growth
for $\sigma$. This suggests that both approaches to this problem are useful. In
the case where $\alpha=1$, our result allows for $b$ and $\sigma$ that both grow
faster than those in~\cite{chen.huang:23:superlinear}.

\begin{table}[htpb!]
  \centering

  \renewcommand{\arraystretch}{1.7}
  \begin{subtable}{\linewidth}
    \centering
    \begin{tabular}{|c|c|}
      \hline
      $h(u)\sim$                   & $\sigma(u)$ can grow as fast as                                  \\ \hline
      $u\log(u)$                   & $u (\log u)^{1/4}\left(\log(u,2)\right)^{-1/2}$                  \\
      $u\log(u)\log(u,2)$          & $u (\log u)^{1/4}(\log(u,2) )^{-1/4}$                            \\
      $u\log(u)\log(u,2)\log(u,3)$ & $u (\log u)^{1/4}(\log(u,2))^{-1/4}\left(\log(u,3)\right)^{1/4}$ \\ \hline
    \end{tabular}
    \caption{The case when $d=1$ and the noise is the space-time white noise: $\alpha = 1/2$.}
    \label{Tb:ST-White-1D}
  \end{subtable}
  \bigskip

  \begin{subtable}{\linewidth}
    \centering
    \begin{tabular}{|c|c|}
      \hline
      $h(u)\sim$                   & $\sigma(u)$ can grow as fast as                \\ \hline
      $u\log(u)$                   & $u(\log u)^{1/2}\left(\log(u,2)\right)^{-1/2}$ \\
      $u\log(u)\log(u,2)$          & $u(\log u)^{1/2}$                              \\
      $u\log(u)\log(u,2)\log(u,3)$ & $u(\log u)^{1/2}(\log(u,3))^{1/2}$             \\ \hline
    \end{tabular}
    \caption{The case when the noise has a bounded correlation function, i.e., $f(0)<\infty$: $\alpha = 1$.}
    \label{Tb:Bdd}
  \end{subtable}
  \bigskip

  \begin{subtable}{\linewidth}
    \centering
    \begin{tabular}{|c|c|}
      \hline
     $h(u)\sim$                                  & $\sigma(u)$ can grow as fast as                                                                   \\ \hline
     $\displaystyle u \prod_{k=1}^{K} \log(u,k)$ & $\displaystyle u \left(\log(u,2)\right)^{-1/2} \prod_{k=1}^{K} \left(\log(u,k)\right)^{\alpha/2}$ \\ \hline
    \end{tabular}
    \caption{The general case: $\alpha\in(0,1]$ with $K\ge 1$.}
    \label{Tb:General}
  \end{subtable}

  \caption{The growth rate of $\sigma$ is listed in the second column. The $h$'s
  in the first column are typical examples that satisfy the Osgood condition
(see part 3-(b) of Assumption~\ref{A:Osgood}). The growth rate of $\sigma$ is
listed in the second column. }

  \label{Tb:Eg}

\end{table}

\bigskip This paper is organized as follows. In Section~\ref{S:Prelim}, we
introduce some notation and establish some technical results. Our main
result---Theorem~\ref{T:Main}---is proved in Section~\ref{S:Main}.
Theorem~\ref{T:Ex-Blowup} is proved in Section~\ref{S:Example}. Finally, we
recall the Burkholder-Davis-Gundy inequality for the martingale in Banach space
in the Appendix~\ref{S:BDG}.

\section{Some preliminaries}\label{S:Prelim}

In the following, $\Norm{\cdot}_{L^p}$ refers to
$\Norm{\cdot}_{L^p\left(\R^d\right)}$ with $p\in [1,\infty]$ and $\Norm{X}_p
\coloneqq \E\left(|X|^p\right)^{1/p}$. \medskip

We will often use the following simple but useful property of $V_p$, which shows
that the space $V_p$ is monotone in $p$.
\begin{lemma}\label{L:Lr}
  If $1\le p \le r<\infty$, then $V_p \subseteq V_r \subseteq L^r(\mathbb{R}^d)$
  and for any $v \in V_p$,
  \begin{equation}\label{E:Lr}
    \Norm{v}_{L^r} \le \Norm{v}_{V_r} \leq \Norm{v}_{V_p}.
  \end{equation}
\end{lemma}
\begin{proof}
  The proof is straightforward. Let $r \in [p,+\infty)$. Then notice that
  $\Norm{v}_{L^r}^r = \int_{\mathbb{R}^d} |v(x)|^r \ud x$, which is less than
  $\left(\int_{\mathbb{R}^d} |v(x)|^p \ud x \right) \sup_{x \in \mathbb{R}^d}
  |v(x)|^{r-p} = \Norm{v}_{L^p}^p\Norm{v}_{L^\infty}^{r-p} \leq
  \Norm{v}_{V_p}^r$.
\end{proof}

\begin{lemma}\label{L:Stable}
  Assume that $b$ and $\sigma$ in~\eqref{E:SHE} satisfy
  Assumption~\ref{A:Osgood}. For any $p\ge 1$, if $v \in V_p$, then the
  compositions $f(v) \in V_p$ and $\sigma(v) \in V_p$. Moreover,
  \begin{equation}\label{E:Stable}
    \Norm{b(v)}_{V_p} \leq h\left(\Norm{v}_{V_p}\right) \quad \text{and} \quad
    \Norm{\sigma(v)}_{V_p} \le \Norm{v}_{V_p}^{1-\alpha/2} h\left(\Norm{v}_{V_p}\right)^{\alpha/2} \left(\log\left(\frac{h(\Norm{v}_{V_p})}{\Norm{v}_{V_p}}\right)\right)^{-1/2}.
  \end{equation}
\end{lemma}

\begin{proof}
  We first prove the $L^\infty(\R^d)$ norm. We claim that
  \begin{equation}\label{E:bsigma-inf}
    \begin{aligned}
      \Norm{b(v)}_{L^\infty}      & \leq h\left(\Norm{v}_{L^\infty}\right) \quad \text{and} \\
      \Norm{\sigma(v)}_{L^\infty} & \leq \Norm{v}_{L^\infty}^{1-\alpha /2} h\left(\Norm{v}_{L^\infty}\right)^{\alpha/2} \left(\log\left(\frac{h(\Norm{v}_{L^\infty})}{\Norm{v}_{L^\infty}}\right)\right)^{-1/2}.
    \end{aligned}
  \end{equation}
  Since $\left|b\left(v(x)\right)\right|\leq |h(v(x))|$ for all $x\ge 0$, the
  first inequality in~\eqref{E:bsigma-inf} is proved by taking supremum on both
  sides of this inequality proves. As for the second inequality
  in~\eqref{E:bsigma-inf}, since $x \mapsto \frac{h(x)}{x}$ is nondecreasing,
  without loss of generality we may assume that $g(x)\coloneqq \frac{h(x)}{x}$
  is differentiable. By Assumption~\ref{A:Osgood}, we see that the function
  $F(x)\coloneqq \frac{g(x)^{\alpha/2}}{\sqrt{\log\left(g(x)\right)}}$ is
  nondecreasing for $x\ge 0$ since
  \begin{align*}
     F'(x) = \frac{g(x)^{-1+\alpha/2} g'(x)}{2\alpha \left[\log\left(g(x)\right)\right]^{3/2}}
    \times \left(\log\left(g(x)\right) - \frac{1}{\alpha}\right)
    \ge 0.
  \end{align*}
  This proves the second inequality in~\eqref{E:bsigma-inf}.

  The interesting part of the proof is showing that the $L^p (\R^d)$ norm is
  bounded. To this end, observe that
  \begin{equation*}
    \Norm{b(v)}^p_{L^p }
    = \int_{\R^d} |b\left(v(x)\right)|^p \ud x
    \leq \int_{\R^d} |v(x)|^p \left(\frac{h\left(|v(x)|\right)}{|v(x)|}\right)^p \ud x
    \leq \Norm{v}_{L^p }^p \sup_{x \in \R^d} \left(\frac{h\left(|v(x)|\right)}{|v(x)|}\right)^p,
  \end{equation*}
  where the first inequality is obtained by using the bound on $b$ given in Assumption \ref{A:Osgood}.
  From the assumption that $v \mapsto \frac{h(v)}{v}$ is increasing, the above
  display is bounded by
  \begin{align}
    \leq \Norm{v}_{L^p }^p \left(\frac{h\left(\Norm{v}_{L^\infty}\right)}{\Norm{v}_{L^\infty}}\right)^p
    \leq \Norm{v}_{V_p}^p \left(\frac{h(\Norm{v}_{V_p})}{\Norm{v}_{V_p}}\right)^p.
  \end{align}
  Therefore, we can conclude that $\Norm{b(v)}_{L^p } \leq h(\Norm{v}_{V_p})$.
  Combining this with the first relation in~\eqref{E:bsigma-inf} proves the
  first inequality in~\eqref{E:Stable}. The argument for the case of $\sigma$ is
  similar and one needs to use Assumption \ref{A:Osgood} and the following inequality:
  \begin{equation}
    \Norm{\sigma(v)}_{L^p }^p
    \le \int_{\R^d} |v(x)|^p \left(\frac{h\left(|v(x)|\right)}{|v(x)|}\right)^{\alpha p/2}\left(\log\left(\frac{h(|v(x)|)}{|v(x)|}\right)\right)^{-p/2} \ud x.
  \end{equation}
  The rest of the arguments are the same as those for $b$.
\end{proof}

\begin{lemma}\label{L:Iterate}
  (1) Let $p\ge 1$. If for some $T, M>0$, $\Psi:\R_+\times\R^d \to \R$ satisfies
  $ \displaystyle \sup_{t\in[0,T]} \Norm{\Psi(t,\cdot)}_{V_p}\le M<\infty$, then
  \begin{equation}\label{E:Iterate-b}
    \Norm{\int_0^t \int_{\mathbb{R}^d} p_{t-s}(\cdot-y) \Psi(s,y)\ud y\ud s}_{V_p}
     \leq t M.
  \end{equation}
  (2) Let $\Phi:\Omega\times\R_+\times\R^d\to \R$ be an adapted and jointly
  measurable random field. Suppose that $p\ge 2$. If for any $T>0$, there exists
  a constant $M>0$ such that $\sup_{t \in [0,T]} \Norm{\Phi(t,\cdot)}_{V_p} \leq
  M$ a.s., then for all ${k> \max\{(2 + d)/\alpha,p\}}$, there exists a constant
  $C>0$ depending only on $(d,\alpha, p)$, but not on $(T,M,k)$, such that
  \begin{equation}\label{E:Iterate-s}
    \E\left(\sup_{t \in [0,T]} \Norm{\int_0^t \int_{\mathbb{R}^d} p_{t-s}(\cdot-y) \Phi(s,y)W(\ud s,\ud y)}_{V_p}^k\right)
    \leq {C^k k^{k/2}} T^{(\alpha k-d)/2} (1 + T^{d/2}) M^k.
  \end{equation}
\end{lemma}
\begin{proof}
  Part~(1) is obtained by an application of the Minkowski inequality. Part~(2)
  will be proved in three steps. Denote the stochastic integral by $Z(t,x)$. By
  the factorization lemma (see~\cite[Section 5.3.1]{da-prato.zabczyk:14:stochastic}), for $\beta \in
  (0,\alpha)$,
  \begin{align} \label{E:factor-Z}
    Z(t,x) = & \frac{\sin\left(\beta\pi/2\right)}{\pi} \int_0^t \int_{\mathbb{R}^d} (t-r)^{-1+ \beta/2} Y(r,z) p_{t-r}(x-z) \ud z \ud r, \\
    Y(r,z) = & \int_0^r \int_{\R^d} (r-s)^{-\beta/2} p_{r-s}(z-y) \Phi(s,y) W(\ud s,\ud y). \label{E:factor-Y}
  \end{align}
  In the following, we use $C$ to denote a generic constant that does not depend
  on $T$, $M$, $k$ and $p$, whose value may change at each appearance. \medskip

  \noindent\textbf{Step I.~} In this step, we will show that for all
  $k>\max\left((2+d)/\beta, p\right)$, $p\ge 2$, and $T>0$, it holds that
  \begin{align}\label{E:Step-1}
    \E\left(\sup_{(t,x)\in [0,T]\times\R^d} |Z(t,x)|^k\right)
     & \leq {C^k k^{k/2}} M^k {T^{(\alpha k -d)/2}}.
  \end{align}
  By H\"older inequality with exponents $k$ and $\frac{k}{k-1}$, for arbitrary
  $t>0$ and $x \in \mathbb{R}^d$,
  \begin{align*}
    |Z(t,x)|^k \leq C^k\left(\int_0^t \int_{\mathbb{R}^d} (t-s)^{\frac{(\beta/2 -1)k}{k-1}}|p_{t-s}(x-y)|^{\frac{k}{k-1}} \ud y\ud s\right)^{k-1} \int_0^t \int_{\mathbb{R}^d}|Y(s,y)|^k\ud y\ud s.
  \end{align*}
  We use the fact that
  \begin{align*}
    |p_{t-s}(x-y)|^{\frac{k}{k-1}} = p_{t-s}(x-y)|p_{t-s}(x-y)|^{\frac{1}{k-1}}\leq C(t-s)^{-\frac{d}{2(k-1)}}p_{t-s}(x-y),
  \end{align*}
  along with the fact that $p_{t-s}(\cdot)$ is a density to bound the above
  expression by
  \begin{align*}
    |Z(t,x)|^k \leq \left(\int_0^t (t-s)^{\frac{(\beta/2 -1)k}{k-1} - \frac{d}{2(k-1)}}\ud s\right)^{k-1} \int_0^t \int_{\mathbb{R}^d}|Y(s,y)|^k\ud y\ud s.
  \end{align*}
  Notice that
  \begin{align*}
    \frac{(\beta/2 -1)k}{k-1} - \frac{d}{2(k-1)} >-1
    \quad \Longleftrightarrow \quad
    k> \frac{2 + d}{\beta}.
  \end{align*}
  Hence, if we choose $k> (2+d)/\beta$, then the first integral is finite and
  \begin{align}\label{E:Z-bound}
    |Z(t,x)|^k &\leq C t^{\frac{\beta k}{2} -\frac{d}{2} -1} \int_0^t \int_{\mathbb{R}^d}|Y(s,y)|^k \ud y \ud s.
  \end{align}
  Hence,
  \begin{align*}
    \E\left(\sup_{(t,x)\in [0,T]\times\R^d} |Z(t,x)|^k\right)
    & \le C T^{\frac{\beta k}{2} -\frac{d}{2} -1} \int_0^T\ud r\int_{\R^d}\ud z\: \E\left(\left|Y(r,z)\right|^k\right).
  \end{align*}
  Now, for $Y(r,z)$, by the Burkholder-Davis-Gundy inequality, we see that
  \begin{align} \label{E:BDG-Y}
     \Norm{Y(r,z)}_k^2 \le
     {8k} \int_0^r\ud s \iint_{\R^{2d}}\ud y\ud y'\: (r-s)^{-\beta} f(y-y')
          \:  p_{r-s}(y) \Norm{\Phi(s,z-y)}_k   &  \nonumber \\
      \times  p_{r-s}(y') \Norm{\Phi(s,z-y')}_k & \:.
  \end{align}
  Then by the Minkowski inequality, we see that
  \begin{align*}
    \int_{\R^d} \E\left(\left|Y(r,z)\right|^k\right) \ud z
      \le & {C^k k^{k/2}} \int_{\R^d} \Bigg(
              \int_0^r\ud s \iint_{\R^{2d}}\ud y\ud y'\: (r-s)^{-\beta} f(y-y')\\
          & \quad \quad \quad \times p_{r-s}(y)p_{r-s}(y')
              \Norm{\Phi(s,z-y)}_k \Norm{\Phi(s,z-y')}_k
            \Bigg)^{k/2} \ud z \\
      \le & {C^k k^{k/2}} \Bigg(
              \int_0^r\ud s \iint_{\R^{2d}}\ud y\ud y'\: (r-s)^{-\beta} f(y-y')\\
          & \quad \quad \quad \times p_{r-s}(y)p_{r-s}(y')
              \Norm{\Norm{\Phi(s,\cdot-y)}_k \Norm{\Phi(s,\cdot-y')}_k}_{L^{k/2}}
            \Bigg)^{k/2}\\
      \le & {C^k k^{k/2}} \Bigg(
              \int_0^r\ud s \iint_{\R^{2d}}\ud y\ud y'\: (r-s)^{-\beta} f(y-y') \\
          & \quad \quad \quad \times p_{r-s}(y)p_{r-s}(y')
              \Norm{\Norm{\Phi(s,\cdot)}_k}_{L^{k}}^2
            \Bigg)^{k/2},
  \end{align*}
  where in the last inequality we applied the H\"older inequality. If $k\ge p$,
  then we can use the Fubini theorem and the assumption on $\Phi(\cdot, \circ)$
  to obtain that
  \begin{align*}
    \Norm{\Norm{\Phi(s,\cdot)}_k}_{L^{k}}^k
    & =   \E\left(\int_{\R^d}|\Phi(s,z)|^k\ud z\right)                                             \\
    & \le \E\left(\left(\int_{\R^d}|\Phi(s,z)|^p\ud z\right)\Norm{\Phi(s)}_{L^\infty}^{k-p}\right) \\
    & \le \E\left(\sup_{s\in[0,T]}\Norm{\Phi(s,\cdot)}_{V_p}^k\right) \leq M^k,
  \end{align*}
  for all $s\in [0,T]$, Hence,
  \begin{align*}
    \int_{\R^d} \E\left(\left|Y(r,z)\right|^k\right) \ud z
      \le & {C^k k^{k/2}} M^k \left(
        \int_0^r\ud s \iint_{\R^{2d}}\ud y\ud y'\: (r-s)^{-\beta} f(y-y')
        p_{r-s}(y)p_{r-s}(y')
      \right)^{k/2}.
  \end{align*}
  By the Plancherel theorem, we see that
  \begin{align*}
       \int_0^r  \ud s \iint_{\R^{2d}}\ud y\ud y'\: (r-s)^{-\beta} f(y-y') p_{r-s}(y)p_{r-s}(y')
    = & (2\pi)^{-d}\int_0^r\ud s\: (r-s)^{-\beta} \int_{\R^d}e^{-(r-s) |\xi|^2}\widehat{f}(\xi)\ud \xi \\
    = & (2\pi)^{-d}\int_0^r\ud s\: s^{-\beta} \int_{\R^d}e^{-s |\xi|^2}\widehat{f}(\xi)\ud \xi.
  \end{align*}
  By Assumption~\ref{A:noise} and the fact that the function $s\to
  \int_{\R^d}e^{-s |\xi|^2}\widehat{f}(\xi) \ud \xi$ is non-increasing, we see
  that for some universal constant $C>0$,
  \begin{align*}
    \int_{\R^d}e^{-s |\xi|^2}\widehat{f}(\xi) \ud \xi \le  C s^{-(1-\alpha)},
    \quad \text{for all $s>0$}.
  \end{align*}
  Hence,
  \begin{align}\label{E:quad-var}
    \int_0^r\ud s \iint_{\R^{2d}}\ud y\ud y'\: (r-s)^{-\beta} f(y-y') p_{r-s}(y)p_{r-s}(y')
    \leq C \int_0^r  s^{-\beta+\alpha-1}\ud s
    = C r^{\alpha-\beta}.
  \end{align}
  Thus, we have that
  \begin{align}\label{E:EYp}
    \int_{\R^d} \E\left(\left|Y(r,z)\right|^k\right) \ud z
    \le {C^k k^{k/2}} M^k {r^{(\alpha-\beta)k/2}}.
  \end{align}
  Finally, if $k> \max\left((2+d)/\beta,p\right)$, by putting~\eqref{E:EYp} back
  into~\eqref{E:Z-bound}, we prove the claim in~\eqref{E:Step-1}. \bigskip

  \noindent\textbf{Step II.~} In this step, we will show that for all
  $k>\max\left((2+d)/\beta, 2\right)$, $p\ge 2$, and $T>0$,
  it holds that
  \begin{align}\label{E:Step-2}
    \E\left(\sup_{t\in [0,T]} \Norm{Z(t,\cdot)}_{L^p}^k\right)
    \le {C^k k^{k/2}} M^k T^{\alpha k/2}.
  \end{align}
  By the Minkowski inequality and the H\"older inequality, we see that for
  $t\in[0,T]$,
  \begin{align*}
    \Norm{Z(t,\cdot)}_{L^p}
    & \le \int_0^t \ud r \: (t-r)^{-1 + \beta /2} \int_{\R^d} \ud z\: p_{t-r}(z) \Norm{Y(r,\cdot-z)}_{L^p} \\
    & = \int_0^t (t-r)^{-1 + \beta /2} \Norm{Y(r,\cdot)}_{L^p} \ud r                                       \\
    & \le \left(\int_0^t (t-r)^{\frac{k}{k-1}\left(-1 + \beta /2\right)}\ud r\right)^{\frac{k-1}{k}} \left(\int_0^t \Norm{Y(r,\cdot)}_{L^p}^k \ud r\right)^{1/k}.
  \end{align*}
  If $k> (2+d)/\beta$, the above $\ud r$--integral is finite and hence,
  \begin{align}\label{E:factored-moment}
    \E\left(\sup_{t\in[0,T]}\Norm{Z(t,\cdot)}_{L^p}^k\right)
    \le & {C^k} T^{-1+\beta k/2}\int_0^T \E\left(\Norm{Y(r,\cdot)}_{L^p}^k\right) \ud r.
  \end{align}
  To estimate $\E\left(\Norm{Y(r,\cdot)}_{L^p}^k\right)$, if we assume that
  $k,p\ge 2$, then we can apply the BDG inequality in Lemma~\ref{L:Banach-BDG}
  to get
  \begin{align*}
       & \E\left(\|Y(r,\cdot)\|^k_{L^p}\right)                                                                                                                                                                                                              \\
  \leq & C^k k^{\frac{k}{2}} \E \left[\left(\int_0^t \left[\int_{\R^d} \left(\iint_{\R^{2d}} (t-s)^{-\beta}p_{t-s}(x-y)p_{t-s}(x-y')\Phi(s, y)\Phi(s,y')f(y-y')\ud y  \ud y '\right)^{\frac{p}{2}} \ud x  \right]^{\frac{2}{p}} \ud s \right)^{k/2}\right]  \\
  =    & C^k k^{\frac{k}{2}} \E \left[\left(\int_0^t (t-s)^{-\beta} \left[\int_{\R^d} \left(\iint_{\R^{2d}} p_{t-s}(y)p_{t-s}(y')\Phi(s, x-y)\Phi(s,x-y')f(y-y')\ud y  \ud y '\right)^{\frac{p}{2}} \ud x  \right]^{\frac{2}{p}} \ud s \right)^{k/2}\right] \\
  \leq & C^k k^{\frac{k}{2}} \E \left[\left(\int_0^t \iint_{\R^{2d}} (t-s)^{-\beta}p_{t-s}(y)p_{t-s}(y')\left\|\Phi(s, \cdot-y)\Phi(s,\cdot-y')\right\|_{L^{p/2}}f(y-y')\ud y  \ud y ' \ud s \right)^{k/2}\right]                                           \\
  \leq & C^k k^{\frac{k}{2}} \E \left[\left(\int_0^t \iint_{\R^{2d}} (t-s)^{-\beta}p_{t-s}(y)p_{t-s}(y')\left\|\Phi(s, \cdot)\right\|^2_{L^{{p}}}f(y-y')\ud y  \ud y ' \ud s \right)^{k/2}\right]\,.
    \end{align*}
  Then based on the assumption that $\Norm{\Phi(s,\cdot)}_{L^p} \leq M$ a.s.,
  and thanks to~\eqref{E:quad-var}, we see that
  \begin{equation*}
    \E\left(\Norm{Y(r,\cdot)}_{L^p}^k\right)
    \leq C^k k^{k/2} M^k r^{(\alpha-\beta)k/2}.
  \end{equation*}
  Combining the above estimate with~\eqref{E:factored-moment}
  proves~\eqref{E:Step-2}. \bigskip

  \medskip\noindent\textbf{Step III.~} Finally, combining the results from the
  previous two steps shows that if $k>\max\left((2+d)/\beta,p\right)$ and $p\ge
  2$, then for all $T>0$,
  \begin{align*}
    \E\left(\sup_{t\in[0,T]}\Norm{Z(t,\cdot)}_{V_p}^k\right)
    \le & \E\left(\sup_{t\in[0,T]}\left(\Norm{Z(t,\cdot)}_{L^{p}}^k + \Norm{Z(t,\cdot)}_{L^{\infty}}^k\right)\right)                         \\
    \le & C^k\E\left(\sup_{t\in[0,T]}\Norm{Z(t,\cdot)}_{L^{p}}^k\right) + C^k\E\left(\sup_{t\in[0,T]}\Norm{Z(t,\cdot)}_{L^{\infty}}^k\right) \\
    \le & {C^k k^{k/2}} M^k  \left(T^{(\alpha k-d)/2} + T^{\alpha k/2}\right).
  \end{align*}
  This completes the proof of Lemma~\ref{L:Iterate}.
\end{proof}

We next give the exponential estimates for the stochastic integral in the
previous lemma, which will be used in the proof of our main theorem. The
space-time white noise case has been considered by Athreya, \textit{et
al.}~\cite{athreya.joseph.ea:21:small}; see
also~\cite{mueller:91:on,cerrai.rockner:04:large,khoshnevisan:14:analysis}.

\begin{lemma}[Exponential estimates]\label{L:exponential-estimate}
  Assume that $\Phi:\Omega\times\R_+\times\R^d\to\R$ be an adapted and jointly
  measurable random field and $p\ge 2$. If for any $T\ge 0$, there exists a
  constant $M = M(T,p) \ge 0$ such that
  \begin{align*}
    \sup_{t \in [0,T]} \Norm{\Phi(t,\cdot)}_{V_p} \leq M, \quad \text{a.s.},
  \end{align*}
  then, there exists a constant $C>0$ independent of $M$ and $T$ such that for
  any $\delta>0$,
  \begin{equation}\label{E:exponential-estimate}
    \Pro \left(\sup_{t \in [0,T]} \Norm{ \int_0^t \int_{\mathbb{R}^d} G(t-s,x-y)\Phi(s,y)W(\ud s,\ud y )}_{V_p}>\delta \right)
    \leq C {\left(1 + T^{-d/2}\right)} e^{-C\delta^2M^{-2} T^{-\alpha}}
  \end{equation}
  where $\alpha$ is from Assumption \ref{A:noise}.
\end{lemma}
\begin{proof}
  Denote $Z(t,x) \coloneqq \int_0^t \int_{\mathbb{R}^d} G(t-s,x-y)\Phi(s,y)W(\ud
  s, \ud y)$. Fix an arbitrary $\lambda > 0$. From Taylor series, we see that
  \begin{equation*}
    \E \left[\exp \left( \lambda \sup_{t \in [0,T]} \Norm{Z(t,\cdot)}_{V_p} \right)\right]
    = \sum_{k=0}^\infty \frac{\lambda^k}{k!}\E\left[\sup_{t \in [0,T]} \Norm{Z(t,\cdot)}_{V_p}^k\right].
  \end{equation*}
  We claim that
  \begin{equation}\label{E:AppMom}
    \E \left[\exp \left( \lambda \sup_{t \in [0,T]} \Norm{Z(t,\cdot)}_{V_p} \right)\right]
    \leq \left(1 + T^{d/2}\right)T^{-d/2}\sum_{k=0}^\infty \frac{\lambda^k C^k M^k T^{k\alpha/2} k^{k/2}}{k!}.
  \end{equation}
  Indeed, the above inequality~\eqref{E:AppMom} follows from the moment
  estimates in part~(2) of Lemma~\ref{L:Iterate} for $k> \max\left(p,
  (d+2)/\alpha\right)$. If $k\le \max\left(p, (d+2)/\alpha\right)$, one can use
  the Jensen inequality and then pick the leading constant big enough. This
  proves the claim in~\eqref{E:AppMom}.

  In order to transform the summation in~\eqref{E:AppMom} into an exponential
  form, we apply Stirling's approximation $k! \sim \left(k/e\right)^k \sqrt{2
  \pi k}$ to see that
  \begin{equation*}
    \frac{k^{k/2}}{k!} \times \Gamma\left(k/2 + 1\right)
    \sim \frac{1}{\sqrt{2}} \left(\sqrt{e/2}\right)^{k},
  \end{equation*}
  which implies that for some universal constant $\Theta>0$,
  \begin{align*}
    \frac{k^{k/2}}{k!} \le \frac{\Theta^{k+1}}{\Gamma\left(k/2 + 1\right)},
    \quad \text{for all $k = 0,1,2,\cdots$}.
  \end{align*}
  Therefore, from~\eqref{E:AppMom}, we see that
  \begin{align*}
    \E \left[\exp \left( \lambda \sup_{t \in [0,T]} \Norm{Z(t,\cdot)}_{V_p} \right)\right]
    & \leq \left(1 + T^{-d/2}\right) \Theta \sum_{k=0}^\infty \frac{\Theta^kC^k\lambda^k M^k T^{k\alpha/2}}{\Gamma\left(k/2+1\right)} \notag \\
    & = \left(1 + T^{-d/2}\right) \Theta \exp\left(\Theta^2C^2\lambda^2 M^2 T^{\alpha}\right)\left[ 1+ \text{Erf}\left(\Theta C \lambda M T^\alpha\right)\right],
  \end{align*}
  where $\text{Erf}(\cdot)$ is the error function and the equality can be found,
  e.g., in Formula~7.2.6 in~\cite{olver.lozier.ea:10:nist}. Notice that
  $\text{Erf}\left(x\right)\le 1$. Hence, we have that
  \begin{align}\label{E:exp-moment}
    \E \left[\exp \left( \lambda \sup_{t \in [0,T]} \Norm{Z(t,\cdot)}_{V_p} \right)\right]
    \le 2 \Theta \left(1 + T^{-d/2}\right)  \exp\left(\Theta^2C^2\lambda^2 M^2 T^{\alpha}\right),
  \end{align}
  Finally, we can derive the exponential tail estimates by the Chebyshev
  inequality: for any $\lambda>0$,
  \begin{align*}
    \MoveEqLeft \Pro\left(\sup_{t \in [0,T]} \Norm{Z(t,\cdot)}_{V_p} > \delta \right)                                                                                                                        \\
    & = \Pro\left(\sup_{t \in [0,T]} \exp\left( \lambda \Norm{Z(t,\cdot)}_{V_p}\right) > \exp(\lambda \delta) \right)                                                                                        \\
    & \leq \exp \left(-\lambda \delta \right) \E \left[\exp\left( \lambda \sup_{t \in [0,T]} \Norm{Z(t,\cdot)}_{V_p} \right)\right]                                                                          \\
    & \leq 2\Theta\left(1 + T^{-d/2}\right) \exp\left( \Theta^2C^2\lambda^2 M^2 T^{\alpha}- \lambda \delta \right)                                                                                           \\
    & = 2\Theta\left(1 + T^{-d/2}\right) \exp\left( \Theta^2C^2 M^2 T^{\alpha} \left(\lambda - \frac{\delta}{2 \Theta^2 C^2 M^2 T^{\alpha}}\right)^2 - \frac{\delta^2}{4\Theta^2 C^2 M^2 T^{\alpha}} \right) \\
    & \leq 2\Theta\left(1 + T^{-d/2}\right) \exp\left( -\frac{\delta^2}{4\Theta^2 C^2 M^2 T^{\alpha}} \right),
  \end{align*}
  which proves Lemma~\ref{L:exponential-estimate}.
\end{proof}

In the next theorem, we generalize Theorem~1.6
of~\cite{chen.huang:23:superlinear} from the original Dalang
condition~\eqref{E:Dalang} to the weaker condition---Assumption~\eqref{A:noise}.
Part (1) of Theorem~\ref{T:Moment} originates from the moment formula
in~\cite{chen.huang:19:comparison}. Part (2) of Theorem~\ref{T:Moment} shows
that if the initial condition $u_0 \in V_p$ with $p\ge 1$, then for the
equation~\eqref{E:SHE} with both $b$ and $\sigma$ being globally Lipschitz and
vanishing at zero, i.e., $b(0) = \sigma(0)=0$, the solution $u(t,\cdot) \in V_p$
for any $t>0$, a.s.

\begin{theorem}[Moment formulas under Lipschitz condition]\label{T:Moment}
  Assuming Assumption~\ref{A:noise}, and that both $b$ and $\sigma$ are globally
  Lipschitz continuous with Lipschitz coefficients $L_b$ and $L_{\sigma}$,
  respectively. Then we have:
  \begin{enumerate}[(1)]
    \item For any $p\geq 2$,
      \begin{equation}
        \Norm{u(t,x)}_p \leq C (\tau + J_+(t,x))\exp\left(C t \max \left(p^{1/\alpha}L_{\sigma}^{2/\alpha}, L_b\right)\right)\,,
      \end{equation}
      where $J_+(t,x) : = (p_t* |u_0|)(x)$, $$ \tau : = \frac{|b(0)|}{L_b} \vee
      \frac{|\sigma(0)|}{L_{\sigma}}\,, $$ and the constant $C$ does not depend
      on $(t,x,p,L_b, L_{\sigma})$.
    \item If $u_0 \in L^{\infty}(\R^d)\cap L^p(\R^d)$ and assume that $\sigma(0)
      = b(0) = 0$, then for all $t>0$ and $p>\frac{2+d}{\alpha}$,
      \begin{align*}
        \Norm{\sup_{(s,x)\in [0,t]\times\R^d} u(s,x)}_p
        \leq \Norm{u_0}_{L^\infty} + C \Norm{u_0}_{L^p} \left(L_b+L_\sigma\right) \exp\left(Ct\max\left(p^{1/\alpha}L_{\sigma}^{2/\alpha}, L_b\right)\right),
      \end{align*}
      where the constant $C$ does not depend on $(t,x,p,L_b,L_\sigma)$.
  \end{enumerate}
\end{theorem}
\begin{proof}
  Comparing the proofs of parts~(b) and~(c) of Theorem~1.6
  of~\cite{chen.huang:23:superlinear}, we see that one only needs to prove
  part~(1) of the theorem, the proof of which follows a similar argument of that
  used in part (b) of Theorem~1.6 ({\it ibid.}). The proof of part~(2) remains
  unchanged from that of part~(c) of Theorem~1.6 ({\it ibid.}), and so it will
  not be repeated here. Proceeding now to part~(1), according to the proof of
  part~(b) of Theorem~1.6 ({\it ibid.}),
  \begin{align*}
    \Norm{u(t,x)}_p \leq
    \sqrt{3}J_+(t,x) H_{8pL_{\sigma}^2, L_b^2}(t; 1)^{1/2}\,,
  \end{align*}
  where the notation $H_{a,b}(t;1)$ is introduced in Section 2.2
  in~\cite{chen.huang:23:superlinear}. An upper bound of $H_{a,b}(t;1)$ is given
  in Lemma~2.1 in~\cite{chen.huang:23:superlinear}, i.e.,
  \begin{align*}
    \limsup_{t \to \infty}\frac{1}{t} \log H_{a,b}(t; 1)
    \leq & \inf\left\{\beta>0: a\Upsilon (2\beta) + \frac{b}{2\beta^2}< \frac{1}{2}\right\} \\
    \leq & \max\left( \inf \left\{\beta>0: a\Upsilon (2\beta) < \frac{1}{4}\right\}\,, \inf \left\{\beta>0: \frac{b}{2\beta^2} < \frac{1}{4}\right\}\right)\,.
  \end{align*}
  For the first argument in the above maximum, we want to find $\beta$ such that
  \begin{align*}
    \frac{1}{(2\pi)^d} \int_{\R^d} \frac{\widehat{f}(\xi)}{2\beta + |\xi|^2} \ud\xi < \frac{1}{4a}\,.
  \end{align*}
  Notice that
  \begin{align*}
         \MoveEqLeft\frac{1}{(2\pi)^d} \int_{\R^d} \frac{\widehat{f}(\xi)}{2\beta + |\xi|^2} \ud\xi
       = \frac{1}{(2\pi)^d} \int_0^{\infty} \int_{\R^d} e^{-2\beta s} e^{-s|\xi|^2} \widehat{f}(\xi) \ud\xi \ud s \\
     & = \frac{1}{(2\pi)^d} \int_0^{1/\beta} \int_{\R^d} e^{-2\beta s} e^{-s|\xi|^2} \widehat{f}(\xi) \ud\xi \ud s
       + \frac{1}{(2\pi)^d} \int_{1/\beta}^{\infty} \int_{\R^d} e^{-2\beta s} e^{-s|\xi|^2} \widehat{f}(\xi) \ud\xi \ud s \\
     &~\eqqcolon I_1 + I_2.
  \end{align*}
  According to Assumption~\ref{A:noise},
  \begin{gather*}
    I_1 \le C \int_0^{1/\beta} s^{\alpha-1} \ud s = \frac{C}{\beta^{\alpha}}\,, \intertext{and similarly,}
    I_2  =  \int_{\R^d} \frac{e^{-(2\beta+ |\xi|^2)\frac{1}{\beta}}}{2\beta+ |\xi|^2} \widehat{f}(\xi) \ud \xi
        \le \frac{1}{2\beta}\int_{\R^d} e^{-|\xi|^2 \times \frac{1}{\beta}} \widehat{f}(\xi) \ud \xi
        \le \frac{C}{\beta} \left(\frac{1}{\beta}\right)^{\alpha-1} = \frac{C}{\beta^\alpha}.
  \end{gather*}
  Therefore,
  \begin{align*}
    \max\left(
       \inf \left\{\beta>0: a\Upsilon (2\beta) < \frac{1}{4}\right\}\,,
       \inf \left\{\beta>0: \frac{b}{2\beta^2} < \frac{1}{4}\right\}
    \right)
    \leq C \max\left(a^{1/\alpha}, b^{1/2}\right)<\infty\,.
  \end{align*}
  Finally, replacing $a$ and $b$ by $8pL_\sigma^2$ and $L_b^2$, respectively,
  proves part~(1).
\end{proof}

\section{Proof of Theorem~\ref{T:Main}}\label{S:Main}

Now we are ready to prove the main result -- Theorem~\ref{T:Main}.

\begin{proof}[Proof of Theorem~\ref{T:Main}]
  The proof follows the same strategy as that in~\cite{salins:22:global*1}.
  First we define the cutoff functions for $b$ and $\sigma$:
  \begin{equation*}
    b_n(u) \coloneqq \begin{cases}
      b(-3^n) & \text{ if } u<-3^n      \\
      b(u)    & \text{ if } |u|\leq 3^n \\
      b(3^n)  & \text{ if } u>3^n
    \end{cases}
    \quad \text{and} \quad
    \sigma_n(u) \coloneqq \begin{cases}
      \sigma(-3^n) & \text{ if } u<-3^n      \\
      \sigma(u)    & \text{ if } |u|\leq 3^n \\
      \sigma(3^n)  & \text{ if } u>3^n
    \end{cases},\quad \text{respectively.}
  \end{equation*}
  Since both $b_n(\cdot)$ and $\sigma_n(\cdot)$ are globally Lipschitz
  continuous, by part~(2) of Theorem~\ref{T:Moment}, for $p\ge (2+d) /\alpha$,
  there is a unique {$V_p$--valued} solution solving
  \begin{align*}
    u_n(t,x) = & \int_{\mathbb{R}^d} p_t(x-y)u_0(y)\ud y + \int_0^t\int_{\mathbb{R}^d} p_{t-s}(x-y)b_n(u_n(s,y))\ud y \ud s \\
               & + \int_0^t \int_{\mathbb{R}^d} p_{t-s}(x-y)\sigma_n(u_n(s,y)) W(\ud s,\ud y).
  \end{align*}
  Denote the following sequence of stopping times
  \begin{align}\label{E:stopping-times}
    \tau_n \coloneqq \inf\left\{t>0: \: \Norm{u_n\left(t,\cdot\right)}_{V_p}>3^n\right\}.
  \end{align}

  It is easy to check that the solutions are \textit{consistent} in the sense
  that $u_n(t,x) = u_m(t,x)$ for all $t< \tau_n$ whenever $n<m$. We can define a
  \textit{local mild solution} to~\eqref{E:SHE} by setting
  \begin{equation*}
    u(t,x)\coloneqq u_n(t,x) \quad \text{when $t<\tau_n$}.
  \end{equation*}
  This local mild solution will exist until the explosion time
  $\tau_\infty\coloneqq \sup_n \tau_n$. A local solution is called a
  \textit{global solution} if $\tau_\infty = \infty$ with probability one.
  \bigskip

  We build the deterministic sequence
  \begin{equation}\label{E:an-def}
    a_n \coloneqq \min\left\{\frac{\Theta 3^{n+1}}{h\left(3^{n+1}\right)},\frac{1}{n}\right\},
  \end{equation}
  with the constant $\Theta\in {(0,1/3)}$ to be determined later. Just like
  in~\cite{salins:22:global*1}, the Osgood condition $\int_1^\infty
  \frac{1}{h(u)}du = +\infty$ guarantees that
  \begin{equation*}
    \sum_{n = 1}^\infty a_n = +\infty.
  \end{equation*}

  Our goal is to show that the tripling times are bounded below by this
  deterministic sequence $\tau_{n+1} - \tau_n \geq a_n$ for all large $n$, which
  implies that there is a global solution. To this end, we derive the following
  moment estimates. \bigskip

  \noindent\textbf{Claim:~} There exist constants $C>0$ and $q>1$, both
  independent of $n$, such that
  \begin{equation}\label{E:Stopping}
    \Pro \left(\tau_{n+1}-\tau_n < a_n\right) \leq C n^{-q},
    \quad \text{for all $n\in \mathbb{N}$.}
  \end{equation}

  Indeed, as mentioned previously, each $\tau_n$ is well-defined and the
  solution $u(\tau_n,\cdot) \in V_p$. Therefore, we can restart the process at
  time $\tau_n$. For all $t>0$ and $x\in\R^d$, define
  \begin{align*}
    U_n(t,x) & \coloneqq \int_{\mathbb{R}^d} p_t(x-y) u(\tau_n,y)\ud y,                                                              \\
    I_n(t,x) & \coloneqq \int_0^t \int_{\mathbb{R}^d} p_{t-s}(x-y) b(u(\tau_n + s,y))1_{\{s \in [0,\tau_{n+1}-\tau_n]\}}\ud y \ud s, \\
    Z_n(t,x) & \coloneqq \int_0^t \int_{\mathbb{R}^d} p_{t-s}(x-y) \sigma(u(\tau_n+s,y)) 1_{\{s \in [0,\tau_{n+1}-\tau_n]\}} W\left(( \tau_n + \ud s),\ud y\right).
  \end{align*}
  Then for all $t \in [0, \tau_{n+1} - \tau_n]$,
  \begin{align}
    u(\tau_n + t, x) = U_n(t,x) + I_n(t,x)  + Z_n(t,x)
  \end{align}
  Furthermore, the presence of the indicator function $1_{\{s \in
  [0,\tau_{n+1}-\tau_n]\}}$ in the definitions of $I_n(t,x)$ and $Z_n(t,x)$
  guarantees that the integrands are bounded in $V_p$--norm.

  Because $p_t(\cdot)$ is a probability density, it follows from Young's
  inequality for convolutions that for any $t>0$,
  \begin{align}
    \Norm{U_n(t,\cdot)}_{V_p} \leq \Norm{u(\tau_n,\cdot)}_{V_p} = 3^n.
  \end{align}
  Because of the definition of the stopping time $\tau_n$
  in~\eqref{E:stopping-times} and Lemma~\ref{L:Stable}, we see that
  \begin{equation}
    \Norm{b(u(\tau_n + s,y))1_{\{s \in [0,\tau_{n+1}-\tau_n]\}}}_{V_p} \leq h(3^{n+1}).
  \end{equation}
  Therefore, Lemma~\ref{L:Iterate} with $M=h(3^{n+1})$ guarantees that for $t
  \in [0, \tau_{n+1}-\tau_n]$,
  \begin{equation}
    \Norm{I_n(t,\cdot)}_{V_p} \leq t h(3^{n+1}).
  \end{equation}
  In particular, if $t \in [0, a_{n}\wedge (\tau_{n+1} - \tau_n)]$, then by the
  definition of $a_n$ in~\eqref{E:an-def}, we have that
  \begin{equation}
    |I_n(t,\cdot)| \leq a_{n } h(3^{n+1}) \leq 3^n.
  \end{equation}

  The event $\{\tau_{n+1} - \tau_n < {a_{n}}\}$ can only occur if
  $\Norm{u(\tau_n + t,\cdot)}_{V_p} > 3^{n+1}$ for some $t\in(0, a_{n})$. But
  because $\Norm{U_n(t,\cdot)}_{V_p}$ and $\Norm{I_n(t,\cdot)}_{V_p}$ are each
  less than $3^n$ if $t \in [0,{a_{n}}\wedge (\tau_{n+1}- \tau_n)]$, the
  $V_p$--norm can only triple in this short amount of time if the stochastic term
  satisfies
  \begin{align*}
   \sup_{t \in [0,a_{n+1} \wedge (\tau_{n+1} - \tau_n)]}\Norm{Z_n(t,\cdot)}_{V_p} > 3^n.
  \end{align*}
  Because of the definition of the stopping time $\tau_n$,
 ~\eqref{E:stopping-times} and Lemma~\ref{L:Stable}, the $V$--norm of the
  integrand of the stochastic integral is bounded with probability one by
  \begin{equation}
    \left\|\sigma(u(\tau_n+s,y)) 1_{\{s \in [0,\tau_{n+1}-\tau_n]\}} \right\|_{V_p}\leq 3^{(n+1)(1-\alpha/2)}\left( h(3^{n+1})\right)^{\alpha/2} \left(\log \left(\frac{h(3^{n+1})}{3^{n+1}} \right) \right)^{-1/2}
  \end{equation}
  Using the exponential estimate~\eqref{E:exponential-estimate} with
  \begin{align*}
     T = a_n, \quad \delta = 3^n, \quad \text{and} \quad
     M = 3^{(n+1)(1-\alpha /2)} \left[h\left(3^{n+1}\right)\right]^{\alpha /2} \left(\log\left(\frac{h(3^{n+1})}{3^{n+1}}\right)\right)^{-1/2},
  \end{align*}
  we have that
  \begin{align*}
     \Pro\left(\tau_{n+1} - \tau_n < a_n\right)
     & \leq \Pro \left(\sup_{t \in [0,a_n]} \Norm{Z_n(t,\cdot)}_{V_p} \geq 3^n \right)                                                                                                              \\
     & \leq C \left(1+a_n^{-d/2}\right) \exp\left(-\frac{C3^{2(n+1)}\log\left(\frac{h(3^{n+1})}{3^{n+1}} \right)}{3^{2(n+1)} \left(\frac{h(3^{n+1})}{3^{n+1}}\right)^\alpha a_n^\alpha} \right) \\
     & \leq C \left(1+a_n^{-d/2}\right) \exp\left(-C \Theta^{-\alpha}\left\lvert\log\left(a_n/\Theta\right)\right\rvert \right)                                                                 \\
     & \leq Ca_n^{C\Theta^{-\alpha} -d/2}\,.
  \end{align*}
  The second-to-last inequality in the above display is a consequence of the
  definition of $a_n$ and~\eqref{E:an-def}, which guarantees that $a_n\leq
  \frac{\Theta 3^{n+1}}{h(3^{n+1})}$. Now set $q := C \Theta^{-\alpha} -d/2$ and
  choose $\Theta \in (0,1/3)$ small enough so that $q>1$. Then we obtain that
  \begin{equation}
    \Pro\left(\tau_{n+1} - \tau_n < a_n\right) \leq C a_n^{q}.
  \end{equation}
  From the definition of $a_n$ in~\eqref{E:an-def}, we also know that $a_n\le
  1/n$. Therefore,
  \begin{equation}
    \Pro\left(\tau_{n+1} - \tau_n < a_{n}\right) \leq C n^{-q}.
  \end{equation}
  This proves the claim in~\eqref{E:Stopping}. \bigskip

  Finally, we can prove the main result. From the claim in~\eqref{E:Stopping},
  \begin{equation}
    \sum_{n = 1}^\infty \Pro\left(\tau_{n+1} - \tau_n < a_{n}\right)
    \leq C \sum_{n = 1}^\infty n^{-q}
    <+\infty.
  \end{equation}
  By the Borel-Cantelli Lemma, with probability one there exists a random
  $N(\omega)$ such that for all $n \geq N(\omega)$, $\tau_{n+1} - \tau_n \geq
  a_{n}$. Because $\sum a_{n} = +\infty$ this implies that
  \begin{equation}
    \Pro\left(\sup_n \tau_n = +\infty\right) = 1
  \end{equation}
  proving that the solutions cannot explode in finite time. This completes the
  proof of Theorem~\ref{T:Main}.
\end{proof}

\section{An explosion example---the proof of Theorem~\ref{T:Ex-Blowup}}\label{S:Example}

\begin{proof}[Proof of Theorem~\ref{T:Ex-Blowup}]
  We will prove this theorem via contradiction. Fix an arbitrary $p\ge 2$. We assume that
  the conclusion is false, namely, for all $u_0\in V_p$,
  \begin{align}\label{E:Contradiction-Ass}
    \Norm{u(t,\cdot)}_{V_p} <\infty, \quad \text{a.s. for all $t>0$,}
  \end{align}
  and seek a contradiction. For this purpose, it suffices to consider
  the initial condition of the following form, where $p_1(x)$ is the heat kernel from \eqref{E:Mild},
  \begin{align}\label{E:u_0-Theta}
    u_0(x) = \Theta\: p_1(x), \quad \text{for some $\Theta>0$.}
  \end{align}
  It is clear that $u_0\in V_p$. The proof consists of the following two steps.
  \bigskip

  \noindent\textbf{Step~1.~} In this step, we claim that
  under~\eqref{E:Contradiction-Ass}, there exists $\Theta_0>0$ such that
  \begin{align}\label{E:Claim_1}
    \Pro\left( Y_{1/2} \geq 2L \right) >0, \quad \text{for all $L>0$ and $\Theta \ge \Theta_0$}.
  \end{align}
  Indeed, let $t\in(0,1)$. Multiply $p_{1-t}(x)$ on both sides
  of~\eqref{E:Mild} and integrate over $x$ to obtain
  \begin{gather}\label{E:YDM}
    Y_t = Y_0 + D_t + M_t \quad \text{a.s.~for all $t\in(0,1]$,} \shortintertext{where} \notag
    \begin{split}
      Y_t & \coloneqq \int_{\R^d} u(t,x) p_{1-t}(x) \ud x \quad \text{with} \quad Y_0 = (p_1*u_0)(0) = \Theta \left(4\pi\right)^{-d/2}, \\
      D_t & \coloneqq \int_0^t\int_{\R^d} p_{1-s}(y) b\left(u(s,y)\right) \ud s \ud y, \quad \text{and}                                 \\
      M_t & \coloneqq \int_0^t\int_{\R^d} p_{1-s}(y) \sigma\left(u(s,y)\right) W(\ud s, \ud y).
    \end{split}
  \end{gather}
  The boundedness assumption on $\sigma$ ensures that $M_t$ is a martingale.
  Assumption in~\eqref{E:Contradiction-Ass} guarantees that $Y_t$ is well
  defined since
  \begin{align}\label{E:Y-Finite}
    0 \le Y_t\le \Norm{u(t,\cdot)}_{L^\infty} \le \Norm{u(t,\cdot)}_{V_p}<\infty, \quad \text{a.s. for all $t>0$.}
  \end{align}
  Note that the nonnegativity of $Y_t$ comes from the comparison principle
  (see~\cite{gei.manthey:94:comparison, chen.huang:19:comparison}  and references therein), which requires
  conditions such as $\sigma(0)=0$ and $b(0)=0$. However, we will show that
  $X_t\coloneqq \E(Y_t)$ will blow up in at $t=1/2$ provided that $\Theta$ is
  large enough, which then implies the claim in~\eqref{E:Claim_1}.

  It remains to show the blow-up of $X_t$. By treating $p_{1-s}(y)\ud yd\Pro$ as
  a probability measure on $\mathbb{R}^d \times \Omega$, we can apply Jensen's
  inequality to see that
  \begin{align*}
    \E\left(D_t\right)
    \ge \int_0^t\ud s\:  b\left(\E\left[\int_{\R^d}\ud y \: p_{1-s}(y) u(s,y)\right]\right)
     =  \int_0^t\ud s\:  b\left(\E\left[Y_s \right]\right),
  \end{align*}
  from which we obtain the following integral inequality
  \begin{align*}
    X_t \ge \Theta (4\pi)^{-d/2} + \int_0^t b(X_s)\ud s.
  \end{align*}
  Hence, by the finite Osgood condition~\eqref{E:Osgood-Fnt}, for some $Y_0>0$,
  $X_t$ blows up in finite time. By increasing the value of $\Theta$ whenever
  necessary, one can ensure that $X_{1/2} = \infty$. This completes the proof of
  the claim in~\eqref{E:Claim_1}. \bigskip

  \noindent\textbf{Step~2.~} Notice that $Y_t$ in~\eqref{E:YDM} can be
  equivalently written as
  \begin{gather}\label{E:YDM_2}
    Y_t = Y_{1/2} + D_t^* + M_t^* \quad \text{a.s.~for all $t\in(1/2,1]$,}
  \end{gather}
  where the initial condition $Y_{1/2}$ is finite a.s. thanks
  to~\eqref{E:Y-Finite},
  \begin{gather*}
    D_t^* \coloneqq \int_{1/2}^t\int_{\R^d} p_{1-s}(y) b\left(u(s,y)\right) \ud s \ud y  \quad \text{and} \quad
    M_t^* \coloneqq \int_{1/2}^t\int_{\R^d} p_{1-s}(y) \sigma\left(u(s,y)\right) W(\ud s, \ud y).
  \end{gather*}

  \noindent\textbf{Step~2-1.~} As in Step~1, the boundedness assumption on
  $\sigma$ guarantees that $\left\{M_t^*:\:t\ge1/2\right\}$ is a martingale. We
  claim that
  \begin{gather}\label{E:Claim_2}
    \Pro \left( \inf_{t \in [1/2, 1]} M_t^* \leq  -L \Bigg| \mathcal{F}_{1/2}\right)
    \leq \exp\left( -\frac{L^2}{2 C_f \Norm{\sigma}_{L^\infty}^2} \right), \quad \text{a.s.~for all $L>0$,} \shortintertext{where}
    C_f \coloneqq \int_{1/2}^1 \ud s \iint_{\mathbb{R}^{2d}}\ud y_1\ud y_2\: p_{1-s}(y_1)p_{1-s}(y_2) f(y_1-y_2) <\infty.
  \end{gather}
  First note that Assumption~\ref{A:noise} guarantees the finiteness of the
  constant $C_f$. Let $\lambda>0$ be some constant to be chosen later. Since
  $\left\{\exp\left( -\lambda M_t^* \right):\: t\ge 1/2\right\} $ is a
  submartingale, by Doob's submartingale inequality, we see that
  \begin{align}\label{E:Max-Doob}
    &     \Pro\left( \inf_{t \in [1/2,1]} M_t^* \leq -L \Bigg| \mathcal{F}_{1/2}\right)
       =  \Pro\left(\sup_{t \in [1/2,1]  } e^{-\lambda M_t^*} > e^{\lambda L} \Bigg| \mathcal{F}_{1/2}\right) \notag
      \le e^{-\lambda L} \E \left[e^{-\lambda M_1^*} \Bigg| \mathcal{F}_{1/2}\right] \notag\\
    &  =  \E\Bigg(\exp\bigg\{ -\lambda L + \frac{\lambda^2}{2} \int_{1/2}^1\ud s \iint_{\mathbb{R}^{2d}}\ud y_1\ud y_2\: \notag \\
    & \quad\quad\quad  \times p_{1-s}(y_1)p_{1-s}(y_2)\sigma(u(s,y_1))\sigma\left(u(s,y_2)\right)f(y_1-y_2) \bigg\} \bigg| \mathcal{F}_{1/2}\Bigg)\notag\\
    & \le \exp\left( -\lambda L + \frac{1}{2} C_f \lambda^2 \Norm{\sigma}_{L^\infty}^2 \right), \quad \text{a.s.,}
  \end{align}
  where in the last inequality we have used the fact that $f$ is nonnegative.
  Optimizing the constant $\lambda$ in~\eqref{E:Max-Doob} proves the claim
  in~\eqref{E:Claim_2}. \bigskip

  \noindent\textbf{Step~2-2.~} Consider the following deterministic equation,
  \begin{align*}
    \widehat{Y}_t = L + \int_{1/2}^t b\left(\widehat{Y}_s\right) \ud s, \quad \text{for $t\ge 1/2$.}
  \end{align*}
  Since $b(\cdot)$ satisfies the finite Osgood condition~\eqref{E:Osgood-Fnt},
  when $L>0$ is large enough, the solution to the above equation will explode before
  time $1$. In the following, we fix this constant $L$.

  Next for all $t \in [1/2,1]$, another application of Jensen's inequality with
  respect to the measure $p_{1-s}(y)\ud y$ to in the term $D_t^*$
  in~\eqref{E:YDM_2} shows that
  \begin{gather*}
    Y_t \geq  \left(Y_{1/2}+ M_t^*\right) + \int_{1/2}^t b\left(Y_s\right)\ud s, \quad \text{a.s.~for all $t\in(1/2,1]$.}
  \end{gather*}
  Choose and fix arbitrary constant $\Theta >\Theta_0$. We claim that
  \begin{align*}
    \Pro(\Omega_L) > 0 \quad \text{with} \quad
    \Omega_L \coloneqq  \left\{Y_{1/2} + M_t^* \ge L: \: \text{for all $t \in [1/2,1]$}\right\}.
  \end{align*}
  Indeed, by the claims in~\eqref{E:Claim_1} and~\eqref{E:Claim_2}, we see that
  \begin{align*}
    \Pro(\Omega_L)
    & \ge \Pro\left( \left\{ Y_{1/2}\ge 2L \right\} \cap \left\{ \inf_{t\in [1/2,1]} M_t^* > -L\right\} \right)   \\
    & =   \Pro\left( Y_{1/2}\ge 2L \right) \Pro\left( \inf_{t\in [1/2,1]} M_t^* > -L \bigg| Y_{1/2}\ge 2L \right) \\
    & \ge \Pro\left( Y_{1/2}\ge 2L \right) \left(1 - \exp\left( -\frac{L^2}{2 C_f \Norm{\sigma}_{L^\infty}^2} \right)\right)
      >   0.
  \end{align*}
  Therefore,
  \begin{align*}
    Y_t \geq L + \int_{1/2}^t b(Y_s)\ud s, \quad \text{a.s. on $\Omega_L$ for all $t \in [1/2,1]$.}
  \end{align*}
  Since $b(\cdot)$ is nondecreasing, we see that $\widehat{Y}_t$ provides a
  sub-solution to $Y_t$ in the sense that $Y_t\ge \widehat{Y}_t$ a.s.~on
  $\Omega_L$ for all $t\in[1/2,1]$. Hence, with positive probability, i.e.,
  a.s.~on $\Omega_L$ with $\Pro(\Omega_L)>0$, $Y_1\ge \widehat{Y}_1 = \infty$,
  which contradicts with~\eqref{E:Y-Finite}. This completes the proof of
  Theorem~\ref{T:Ex-Blowup}.
\end{proof}

\appendix
\section{Appendix}\label{S:BDG}

In this appendix, we give a result about the Burkholder-Davis-Gundy inequality
for the martingales taking values in a Banach space (typically $L^p(\R^d)$ in
our setting). We begin by introducing some standard concepts, which can be found
in, e.g., \cite{ondrejat:04:uniqueness} or Section~2.2
of~\cite{zhang:10:stochastic}.

\begin{definition}[Definition~3.1 of~\cite{ondrejat:04:uniqueness}]\label{D:2-smooth}
  A Banach space $X$ is said to be \textit{2-smooth} provided there exist an
  equivalent norm $\Norm{\cdot}$ and a constant $C\geq 2$ such that for all $x,
  y \in X$,
  \begin{equation*}
    \|x+y\|^2 + \|x-y\|^2 \leq 2 \|x\|^2 + C \|y\|^2\,.
  \end{equation*}
\end{definition}

\begin{definition}[Definition~2.3 of~\cite{ondrejat:04:uniqueness}]
  Let $H$ be a separable Hilbert space and $X$ a separable Banach space, $B\in
  L(H,X)$ be an operator from $H$ to $X$, and $\xi_i$, $i = 1,2,3,\dots$, be a
  sequence of independent standard Gaussian random variables, and $\left\{e_k:\:
  k=1,2,3,\dots\right\}$ be one orthonormal basis of $H$. Then $B$ is called a
  \textit{radonifying} operator if the series $\sum_{k=1}^{\infty} B(e_k)\xi_k$
  converges in $L^2(\Omega, X)$. The space of radonifying operators are denoted
  by $\gamma\left(H,X\right)$, which is a Banach space with the
  \textit{radonifying norm}
  \begin{equation*}
    \Norm{B}_{\gamma} \coloneqq \left(\E \left[\left\|\sum_{k=1}^{\infty}B(e_k)\xi_k\right\|^2_X\right]\right)^{1/2}\,.
  \end{equation*}
\end{definition}

We have the following two facts (see, e.g., Example~2.9
of~\cite{zhang:10:stochastic} for details):
\begin{enumerate}
  \item The Banach space $X = L^p(\R^d)$ for all $p\in [2,\infty)$ is 2-smooth
    and separable;
  \item For any $B\in \gamma\left(H, L^p\left(\R^d\right)\right)$, it holds that
    \begin{equation}\label{E:Radonfiying-Lp}
      \Norm{B}_{\gamma} \leq C_p \Norm{\:\Norm{B}_H}_{L^p}.
    \end{equation}
\end{enumerate}

The following result about BDG inequality for martingales with values in Banach
space, is from Theorem~1.1 in~\cite{seidler:10:exponential}.

\begin{theorem}\label{T:Banach-BDG}
  Let $X$ be a 2-smooth and separable Banach space with norm $\|\cdot \|_X$, $W$
  be a cylindrical $Q$-Wiener process ($Q$ is the covariance operator) on a real
  separable Hilbert space $H$ and $U = \mathop{range} \left(Q^{1/2}\right)$.
  Then, there exists a constant $\Pi<\infty$, depending only on $\left(X,
  \Norm{\cdot}_X\right)$, such that
  \begin{equation}\label{E:General-B-BDG}
    \Norm{\sup_{0\leq t \leq \tau}\Norm{\int_0^t \psi(s)\ud W(s)}_X}_{k}
    \leq \Pi \sqrt{k} \Norm{\left(\int_0^{\tau} \|\psi(s)\|^2_{\gamma} \ud s \right)^{1/2}}_{k}\,,
    \quad \text{for all $k>2$},
  \end{equation}
  where $\tau$ is a stopping time and $\psi$ is any progressively measurable
  $\gamma(U, X)$-valued stochastic process satisfying
  \begin{align*}
    \int_0^t \|\psi(s)\|^2_{\gamma} \ud s  < \infty \quad \text{for all} \ t\geq 0\ a.s.
  \end{align*}
\end{theorem}

Since the Walsh integral can be written using the setup of $H$-valued process
and the cylindrical Wiener process on $H$, we can apply
Theorem~\ref{T:Banach-BDG} with $X= L^p(\R^d)$, $p\ge 2$, and
combine~\eqref{E:Radonfiying-Lp} and~\eqref{E:General-B-BDG} to get the
following lemma:
\begin{lemma}\label{L:Banach-BDG}
  Let $p\ge 2$ be fixed and $H$ be the Hilbert space introduced
  in~\eqref{E:H-inner-Prod}. Assume that $\psi: \Omega \times \R_+\times \R^d
  \times \R^d\to \R$ be an adapted and jointly measurable random field such that
  \begin{enumerate}
    \item for each $(s,x)\in\R_+\times\R^d$, $\psi(s,\cdot, x) \in H$;
    \item for each $s>0$, $\Norm{\psi(s,\cdot,\circ)}_H \in L^p(\R^d)$.
  \end{enumerate}
  Then, for all $k>2$ and $t>0$, it holds that
  \begin{gather}\label{E:Banach-BDG}
    \Norm{\Norm{\int_0^t \int_{\R^d}\psi(s,y,\circ)W(\ud s, \ud y)}_{L^p}}_k
    \leq C \sqrt{k} \Norm{\left(\int_0^t \Norm{\Norm{\psi(s,\cdot, \circ)}_H}^2_{L^p} \ud s\right)^{1/2}}_{k}\,,
  \end{gather}
  where the constant $C$ does not depend on $k$. Note that
  inequality~\eqref{E:Banach-BDG} is nontrivial only when the right-hand side of
  the inequality is finite.
\end{lemma}

\section*{Acknowledgments}

L.~C. is partially supported by the NSF grant DMS-2246850 and by a collaboration
grant (\#959981) from the Simons foundation. M.S. is partially supported by a collaboration grant from the Simon's foundation (\#962543).

\printbibliography[title={References}]
\end{document}